\documentclass[12pt,a4paper]{article}
\usepackage{amssymb}
\usepackage{amsmath}
\usepackage[symbol]{footmisc}
\usepackage{amsthm}
\usepackage{amscd}
\usepackage{multirow}
\usepackage{amstext}
\usepackage{booktabs}
\usepackage{vmargin}
\usepackage{float}
\usepackage{fancyhdr}
\setpapersize{A4}
\setmarginsrb{3.0cm}{1.5cm}{3.0cm}{3cm}{1cm}{.5cm}{2cm}{2cm}

\usepackage{amsfonts}
\usepackage[numbers]{natbib}
\usepackage[latin1]{inputenc}
\usepackage[T1]{fontenc}
\usepackage[american]{babel}
\usepackage{graphicx}
\usepackage{bbm}
\usepackage[mathscr]{euscript}
\usepackage{mathrsfs}
\usepackage{stmaryrd}
\usepackage{soul}
\usepackage{hyperref}
\usepackage{xcolor}
\usepackage{xspace}
\usepackage[draft,danish]{fixme}
\usepackage{mathtools}
\usepackage{dsfont}
\usepackage{siunitx}
\usepackage{url}
\usepackage[ps,all,dvips]{xy}
\usepackage[shortlabels,inline]{enumitem}
\setlist[enumerate]{font=\textnormal}

\usepackage{caption}

\SelectTips{cm}{12}
\CompileMatrices

  \makeatletter
 \useshorthands{"}
 \defineshorthand{"-}{\nobreak-\bbl@allowhyphens}
 \makeatother

\newcommand{\dd}{\operatorname{d\mkern-2.5mu}}

\let\deg\relax
\DeclareMathOperator{\deg}{deg}

\let\Re\relax
\DeclareMathOperator{\Re}{Re}

\let\Im\relax
\DeclareMathOperator{\Im}{Im}

\theoremstyle{plain}
\newtheorem{lemma}{Lemma}[section]

\newtheorem{theorem}[lemma]{Theorem}
\newtheorem{corollary}[lemma]{Corollary}
\theoremstyle{definition}
\newtheorem{example}[lemma]{Example}

\theoremstyle{definition}

\theoremstyle{remark}
\newtheorem{remark}[lemma]{Remark}

\usepackage{pgfplots}
\pgfplotsset{
	every tick label/.append style={font=\footnotesize}
}

\tikzset{
	invisible/.style={opacity=0},
	visible on/.style={alt={#1{}{invisible}}},
	alt/.code args={<#1>#2#3}{%
		\alt<#1>{\pgfkeysalso{#2}}{\pgfkeysalso{#3}}
	},
}

\pgfplotsset{compat=newest}

\makeatletter \newcommand{\pgfplotsdrawaxis}{\pgfplots@draw@axis} \makeatother

\pgfplotsset{axis line on top/.style={
		axis line style=transparent,
		ticklabel style=transparent,
		tick style=transparent,
		axis on top=false,
		after end axis/.append code={
			\pgfplotsset{axis line style=opaque,
				ticklabel style=opaque,
				tick style=opaque,
				grid=none}
			\pgfplotsdrawaxis}
	}
}

\definecolor{darkgray}{gray}{0.2}

\newcommand{\devnull}[1]{}

\numberwithin{equation}{section}

\title{On non-negative solutions of SDDEs with an application to CARMA processes}
\author{Mikkel Slot Nielsen\footnote{Department of Statistics, Columbia University, USA. E-mail: m.nielsen@columbia.edu.}
	\and Victor Rohde}
\date{\ }

\begin{document}
\maketitle

\begin{abstract}
This note provides a simple sufficient condition ensuring that solutions of stochastic delay differential equations (SDDEs) driven by subordinators are non-negative. While, to the best of our knowledge, no simple non-negativity conditions are available in the context of SDDEs, we compare our result to the literature within the subclass of invertible continuous-time ARMA (CARMA) processes. In particular, we analyze why our condition cannot be necessary for CARMA($p,q$) processes when $p=2$, and we show that there are various situations where our condition applies while existing results do not as soon as $p\geq 3$. Finally, we extend the result to a multidimensional setting.
\end{abstract}

\section{Introduction}
Many quantities, such as wind speeds or (local) volatility of assets, are non-negative and behave in a stationary manner, and thus any reasonable model for these phenomena should comply with such constraints. 
\citet{OEBN_non_G_OU} advocated the use of the stationary Ornstein--Uhlenbeck process driven by a subordinator (or, equivalently, a non-negative Lévy process) $(L_t)_{t\in \mathbb{R}}$, that is, the unique stationary solution to
\begin{equation}\label{OUintro}
\dd X_t = -\lambda X_t \dd t + \dd L_t,\qquad t \in \mathbb{R},
\end{equation}
for some $\lambda \in (0,\infty)$. Since the solution of \eqref{OUintro} is explicitly given by
\begin{equation}\label{MAintroduce}
X_t = \int_{-\infty}^t e^{-\lambda (t-s)}\, \dd L_s,\qquad t \in \mathbb{R},
\end{equation}
 a convolution between a non-negative kernel $t\mapsto e^{-\lambda t}$ and a non-negative random measure $\dd L$, it is automatically non-negative. However, due to the simplicity of the Ornstein--Uhlenbeck process, its autocorrelation function is $\text{Corr}(X_0,X_h)= e^{-\lambda \vert h\vert}$ with the only parameter being the rate of decay $\lambda$; in particular, the autocorrelation function must be monotonically decreasing. Consequently, there has been a need for working with more flexible modeling classes. A particular popular one consists of the continuous-time ARMA (CARMA) processes, which (as the name suggests) is the natural continuous-time analogue to the discrete-time ARMA processes. A CARMA process $(X_t)_{t\in \mathbb{R}}$ can be characterized as a moving average of the form
 \begin{equation}\label{generalMAintro}
 X_t = \int_{-\infty}^t g (t-s)\, \dd L_s,\qquad t \in \mathbb{R},
 \end{equation}
 where the Fourier transform of $g\colon [0,\infty)\to \mathbb{R}$ is rational. Consequently, a CARMA process is made up of a Lévy process $(L_t)_{t \in \mathbb{R}}$, a denominator (autoregressive) polynomial $P$, and a numerator (moving average) polynomial $Q$. Here it is required that $\deg (P)> \deg (Q)$ and that the zeroes of $P$ belong to $\{z\in \mathbb{C}\,:\, \Re (z) < 0\}$. For more details on CARMA processes, see \cite{brockLevy,brockwell2014recent,brockwell2011estimation,marquardt2007multivariate} or Section~\ref{CARMArelations}. Another important class, which also extends the Ornstein--Uhlenbeck process, is the one formed by solutions of stochastic delay differential equations (SDDEs) of the form
 \begin{equation}\label{SDDEintro}
\dd X _t = \int_{[0,\infty)}X_{t-s}\, \phi (\dd s)\, \dd t + \dd L_t,\qquad t \in\mathbb{R},
 \end{equation}
 where $\phi$ is a signed measure. Such equations have, for instance, been studied in \cite{basse2019multivariate,GK,mohammed1990lyapunov}. Under suitable conditions, which will be stated in Section~\ref{SDDEnonNeg}, there exists a unique stationary solution to \eqref{SDDEintro} and it is, as well, a moving average of the form \eqref{generalMAintro} with $g$ being characterized through its Fourier transform. While both CARMA processes and solutions of SDDEs give rise to increased flexibility in the kernel $g$, and hence in the autocorrelation function, it is no longer guaranteed that it stays non-negative on $[0,\infty)$. This means that $(X_t)_{t\in \mathbb{R}}$ is not necessarily non-negative although $(L_t)_{t\in \mathbb{R}}$ is so, and one must instead place additional restrictions on $(P,Q)$ and $\phi$. A necessary and sufficient, but unfortunately rather implicit, condition ensuring non-negativity of $g$ is given by the famous Bernstein theorem on completely monotone functions \cite{bernstein1929fonctions}. In the context of CARMA processes, more explicit sufficient conditions were provided by \cite{ball1994completely,tsai2005note} as they argued that a CARMA process is non-negative if the zeroes of the associated polynomials $P$ and $Q$ are real and negative, and if they respect a certain ordering (see \eqref{CARMAsufficiency}). To the best of our knowledge, this is the only available easy-to-check condition. In addition to the fact that this condition is not able to identify all non-negative CARMA processes, the Fourier transform of the driving kernel $g$ of the solution of an SDDE \eqref{SDDEintro} is most often not rational, meaning that the condition is not applicable in such setting.
 
 In this paper we show that, in order for the unique solution $(X_t)_{t\in \mathbb{R}}$ of \eqref{SDDEintro} to be non-negative, it is sufficient that $\phi$ is a non-negative measure when restricted to~$(0,\infty)$, that is, 
 \begin{equation}\label{ourCond}
 \phi (B\cap (0,\infty)) \geq 0\qquad \text{for all measurable sets $B$.}
 \end{equation}
 This result is presented in Section~\ref{SDDEnonNeg}, in which we also give various examples and illustrate through simulations that simple violations of \eqref{ourCond} results in solutions which can indeed go negative. Furthermore, in Section~\ref{CARMArelations} we exploit the relation between SDDEs and invertible CARMA processes (that is, CARMA processes whose associated moving average polynomial $Q$ only has zeroes in $\{z\in \mathbb{C}\, :\, \Re (z)<0\}$) to establish conditions ensuring that a CARMA process is non-negative. Specifically, we observe that it is sufficient to show complete monotonicity of the rational function $R/Q$ instead of $Q/P$ on $[0,\infty)$, where $R$ is the (negative) remainder polynomial obtained from division of $P$ with $Q$. This result is compared to the findings of \cite{ball1994completely, tsai2005note} and it is shown that the two approaches identify different, but overlapping, regions of parameter values for which the CARMA($3,2$) process is non-negative. Finally, in Section~\ref{multivariate} we extend the theory to the multivariate SDDEs (as introduced in \cite{basse2019multivariate}). In particular, when leaving the univariate setting we find that, in addition to a condition similar to \eqref{ourCond} for a matrix-valued signed measure ${\boldsymbol \phi}$, one needs to require that ${\boldsymbol \lambda} \coloneqq - {\boldsymbol \phi}(\{0\})$ is a so-called $M$-matrix. Section~\ref{proofs} contains proofs of the stated results and a couple of auxiliary lemmas.
 
 Before turning to the above-mentioned sections, we devote a paragraph to introduce relevant notations as well as essential background knowledge.
  
\paragraph*{Preliminaries}
The models that we will consider are built on subordinators or, equivalently, non-negative Lévy processes. Recall that a real-valued stochastic process $(L_t)_{t\geq 0}$, $L_0 \equiv 0$, is called a one-sided subordinator if it has càdlàg sample paths and its increments are stationary, independent, and non-negative. These properties imply that the law of $(L_t)_{t\geq 0}$ (that is, of all its finite-dimensional marginals) is completely determined by that of $L_1$, which is infinitely divisible and, thus,
\begin{equation*}
\log \mathbb{E}[e^{i\theta L_1}] = i\theta\gamma + \int_0^\infty (e^{i\theta x}-1-i \theta x\mathds{1}_{ x  \leq 1})\, \nu (\dd x),\qquad \theta \in \mathbb{R},
\end{equation*}
by the Lévy--Khintchine formula and \cite[Proposition~3.10]{contTankov}. Here $\gamma \in [0,\infty)$ and $\nu$ is a $\sigma$-finite measure on $(0,\infty)$ with $\int_0^\infty (1\wedge x)\, \nu (\dd x)<\infty$ ($\nu$ is a Lévy measure). For any subordinator $(L_t)_{t\geq 0}$ the induced pair $(\gamma,\nu)$ is unique and, conversely, given such pair, there exists subordinator (which is unique in law) inducing this pair. We define a two-sided subordinator $(L_t)_{t\in \mathbb{R}}$ from two independent one-sided subordinators $(L^1_t)_{t\geq 0}$ and $(L^2_t)_{t\geq 0}$ with the same law by
\begin{equation*}
L_t = 
L^1_t\mathds{1}_{[0,\infty)}(t)
 - L^2_{(-t)-} \mathds{1}_{(-\infty,0)}(t),\qquad t \in \mathbb{R}.
\end{equation*}
Examples of subordinators include the gamma Lévy process, the inverse Gaussian Lévy process, and any compound Poisson process constructed from a sequence of non-negative i.i.d.\ random variables. We may also refer to multivariate subordinators, which are simply Lévy processes with values in $\mathbb{R}^d$, $d\geq 2$, and entrywise non-negative increments (see \cite{Sato} for details).

We will say that a real-valued set function $\mu$, defined for any Borel set of $[0,\infty)$, is a signed measure if $\mu  = \mu_+-\mu_-$ for two singular finite Borel measures $\mu_+$ and $\mu_-$ on $[0,\infty)$. Note that the variation $\vert \mu\vert \coloneqq \mu_+ + \mu_-$ of $\mu$ is a measure on $[0,\infty)$ and that integration with respect to $\mu$ can be defined in an obvious manner for any function $f\colon [0,\infty)\to \mathbb{R}$ which is integrable with respect to $\vert \mu \vert$.

In the last part of the paper we will consider a multivariate setting, and to distinguish this from the univariate one, we shall denote a matrix ${\boldsymbol A}$ with bold font and, unless stated otherwise, refer to its $(j,k)$-th entry by $A_{jk}$. Finally, integration of matrix-valued functions against matrix-valued signed measures (that is, matrices whose entries are signed measures) can be defined in an obvious manner by means of the usual rules for matrix multiplication.

\section{Stochastic delay differential equations and non-negative solutions}\label{SDDEnonNeg}
Let $(L_t)_{t\in \mathbb{R}}$ be a two-sided subordinator with a non-zero Lévy measure $\nu$ satisfying $\int_1^\infty  x \, \nu (\dd x)<\infty$ or, equivalently, $\mathbb{E}[L_1]<\infty$. Moreover, let $\phi$ be a signed measure on $[0,\infty)$ with
\begin{equation}\label{phiMoments}
\int_0^\infty t^2 \vert \phi \vert (\dd t)<\infty.
\end{equation}
A stochastic process $(X_t)_{t\in \mathbb{R}}$ is said to be a solution of the corresponding SDDE if it is stationary, has finite first moments (that is, $\mathbb{E}[\vert X_0 \vert]<\infty$), and
\begin{equation}\label{SDDErelation}
X_t - X_s = \int_s^t \int_{[0,\infty)} X_{u-v} \, \phi (\dd v) \, \dd u + L_t -L_s,\qquad s<t.
\end{equation}
By \eqref{SDDErelation}, we mean that the equality holds almost surely for each \emph{fixed} pair $(s,t)\in \mathbb{R}^2$ with $s<t$. We will often write the SDDE in differential form as
\begin{equation*}
\dd X_t = \int_{[0,\infty)} X_{t-s}\, \phi (\dd s)\, \dd t + \dd L_t,\qquad t \in \mathbb{R}.
\end{equation*}

\noindent Set $\mathbb{C}_+\coloneqq \{z \in \mathbb{C}\, :\, \Re(z)\geq 0\}$. Properties such as existence and uniqueness of solutions of \eqref{SDDErelation} are closely related to the zeroes of the function $h_\phi\colon \mathbb{C}_+\to \mathbb{C}$ given by
\begin{equation}\label{hFunction}
h_\phi (z) \coloneqq z- \int_{[0,\infty)} e^{-zt}\, \phi (\dd t),\qquad z \in \mathbb{C}_+.
\end{equation}

\noindent In case $\phi$ has bounded support, existence and uniqueness of solutions of \eqref{SDDErelation} were established in \cite{GK,mohammed1990lyapunov} (under even milder conditions on $(L_t)_{t\in \mathbb{R}}$). As we will later translate our findings into the framework of CARMA processes, which correspond to a particular class of delay measures $\phi$ with unbounded support, we will rely on the following result of \cite{basse2019multivariate}: 

\begin{theorem}[\citet{basse2019multivariate}]\label{existenceTheorem}
	
	Let $h_\phi$ be given as in \eqref{hFunction} and assume that $h_\phi(z) \neq 0$ for all $z \in \mathbb{C}_+$. Then the unique solution of \eqref{SDDErelation} is given by   
	\begin{equation}\label{solution}
	X_t = \int^t_{- \infty} g_\phi(t-u)\, \dd L_u,\qquad t \in \mathbb{R},
	\end{equation}
	where $g_\phi\colon [0,\infty) \to \mathbb{R}$ is characterized by 
	\begin{align}\label{defOfg}
	\int_0^\infty e^{-ity}g_\phi (t)\, \dd t = \frac{1}{h_\phi(iy)},\qquad y \in \mathbb{R}.
	\end{align}
\end{theorem}

\begin{remark}\label{negDelay}
	Although it might sometimes be challenging to show that $h_\phi(z) \neq 0$ when $z\in \mathbb{C}_+$, as required by Theorem~\ref{existenceTheorem}, a necessary and easy-to-check condition is that $\phi ([0,\infty))<0$. Indeed, this follows from the facts that $h_\phi (0) = - \phi ([0,\infty))$, $h_\phi$ is continuous, and $h_\phi (x)\to \infty$ as $x\to \infty$ (where $x\in [0,\infty)$).
\end{remark}

To get an idea of which stationary processes that can be generated from \eqref{SDDErelation}, we provide a couple of examples where the condition of Theorem~\ref{existenceTheorem} can be checked.

\begin{example}[CARMA processes]\label{CARMAprocesses}
	Let $\lambda \in \mathbb{R}$ and, for a given $p\in \mathbb{N}$, suppose ${\boldsymbol b}\in \mathbb{R}^p$ and ${\boldsymbol A}\in \mathbb{R}^{p\times p}$ with a spectrum $\sigma ({\boldsymbol A})$ contained in $\{z \in \mathbb{C}\, :\, \Re(z) >0\}$. Consider the delay measure
	\begin{equation}\label{CARMAdelay}
	\phi (\dd t) = -\lambda\,  \delta_0 (\dd t) + {\boldsymbol b}^\top e^{-{\boldsymbol A}t}{\boldsymbol e}_1\, \dd t,
	\end{equation}
	 where $\delta_0$ is the Dirac measure at $0$ and ${\boldsymbol e}_1$ is the first canonical basisvector of $\mathbb{R}^p$. Note that ${\boldsymbol e}_1$ is merely used as a normalization: the effect of replacing ${\boldsymbol e}_1$ by an arbitrary vector ${\boldsymbol c}\in \mathbb{R}^p$ can be incorporated in ${\boldsymbol b}$ and ${\boldsymbol A}$. By the assumption on $\sigma ({\boldsymbol A})$, all the entries of $e^{-{\boldsymbol A}t}$ are exponentially decaying as $t \to\infty$ and, thus, $\vert \phi\vert$ is a finite measure with moments of any order (in particular, \eqref{phiMoments} is satisfied). The function $h_\phi$ takes the form $h_\phi (z) = z +\lambda - {\boldsymbol b}^\top ({\boldsymbol A}+z{\boldsymbol I}_p)^{-1}{\boldsymbol e}_1$, where ${\boldsymbol I}_p$ is the $p\times p$ identity matrix. By the fraction decomposition it follows that ${\boldsymbol b}^\top ({\boldsymbol A}+z{\boldsymbol I}_p)^{-1}{\boldsymbol e}_1 = R(z)/Q(z)$ for $z\in \mathbb{C}_+$, where $(Q,R)$ is the unique pair of real polynomials $Q,R\colon \mathbb{C}\to \mathbb{C}$ such that (i) $Q$ is monic and has no zeroes on $\mathbb{C}_+$, (ii)~$\text{deg}(Q)>\text{deg}(R)$, and (iii) $Q$ and $R$ have no common zeroes. Consequently, it follows from Theorem~\ref{existenceTheorem} that, as long as
	 \begin{equation*}
	 P(z)\coloneqq (z+\lambda)Q(z) - R(z) \neq 0\qquad \text{for all $z\in \mathbb{C}_+$,}
	 \end{equation*}
	 there exists a unique stationary solution $(X_t)_{t\in \mathbb{R}}$ to \eqref{SDDErelation}, and it is given by \eqref{solution} with $g_\phi$ satisfying 
	 \begin{equation}\label{CARMAkernel}
	 \int_0^\infty e^{-ity}g_\phi (t)\, \dd t = \frac{Q(iy)}{P(iy)} ,\qquad y\in \mathbb{R}.
	 \end{equation}
	 In other words, $(X_t)_{t\in \mathbb{R}}$ is a causal and invertible CARMA process with autoregressive polynomial $P$ and moving average polynomial $Q$ (see Section~\ref{CARMArelations} or \cite[Remark~4]{brockwell2014recent}). Conversely, given a moving average
	 \begin{equation*}
	 X_t = \int_{-\infty}^t g(t-s)\, \dd L_s,\qquad t\in \mathbb{R},
	 \end{equation*}
	 with $g\colon [0,\infty)\to \mathbb{R}$ characterized by \eqref{CARMAkernel} for some polynomials $P$ and $Q$ having no zeroes in $\mathbb{C}_+$, and which satisfy $\text{deg}(P) = \text{deg}(Q) +1$, one can choose a unique constant $\lambda\in \mathbb{R}$ such that the polynomial $R(z)\coloneqq (z+\lambda)Q(z) - P(z)$ meets $\text{deg}(R)<\text{deg}(Q)$. Thus, it follows that such process constitutes a stationary solution of the SDDE \eqref{SDDErelation} with a delay measure $\phi$ of the form \eqref{CARMAdelay}. Indeed, this is due the fact that a function $f\colon [0,\infty)\to \mathbb{R}$ with a rational Laplace transform $R/Q$ can always be represented as $f(t) ={\boldsymbol b}^\top e^{-{\boldsymbol A}t}{\boldsymbol e}_1$ for suitable ${\boldsymbol b}\in \mathbb{R}^{\deg (Q)}$ and ${\boldsymbol A}\in \mathbb{R}^{\deg (Q)\times \deg (Q)}$ with $\sigma ({\boldsymbol A})\subseteq \{z\in\mathbb{C}\, :\, \Re (z)>0\}$. For more on the relation between solutions of SDDEs and CARMA processes, see \cite[Section~4.3]{basse2019multivariate}.
\end{example}

\begin{example}[Discrete delay]\label{discDelay} Let $\lambda,\xi\in \mathbb{R}$ and $\tau \in (0,\infty)$, and assume that $\vert \xi \vert \leq \tau^{-1}$. Consider the following SDDE written in differential form:
	\begin{equation}\label{discDelayEq}
	\dd X_t = (-\lambda X_t + \xi X_{t-\tau})\, \dd t + \dd L_t,\qquad t \in \mathbb{R}.
	\end{equation}
	To show existence of a unique stationary solution using Theorem~\ref{existenceTheorem}, we must argue that $z+ \lambda - \xi e^{-z\tau}\neq 0$ whenever $z\in \mathbb{C}_+$ or, equivalently, that the two equations
	\begin{equation}\label{equationsDisc}
	x+\lambda - \xi e^{-\tau x}\cos (\tau y) = 0\qquad \text{and} \qquad y + \xi \sin (\tau y) = 0
	\end{equation}
	cannot hold simultaneously if $x\geq 0$ and $y\in \mathbb{R}$. Since the only real solution of an equation of the form $u = \sin (\alpha u)$ is $u=0$ when $\vert \alpha \vert \leq 1$, the second equation in \eqref{equationsDisc} implies that $y=0$. If this is the case we have that
	\begin{equation*}
	x + \lambda - \xi e^{-\tau x}\cos (\tau y) \geq \lambda - \xi,
	\end{equation*}
	since $u\mapsto u-\xi e^{-\tau u}$ is increasing on $[0,\infty)$. In view of Remark~\ref{negDelay}, we conclude that $h_\phi (z) \neq 0$ for all $z\in \mathbb{C}_+$ if and only if $\xi<\lambda$. 
\end{example}

Turning to the question of whether a given solution $(X_t)_{t\in \mathbb{R}}$ to \eqref{SDDErelation} is non-negative, we observe initially that, since it will necessarily take the form \eqref{solution}, there exist a number of equivalent statements for non-negativity:
\begin{theorem}\label{nonNegativeOfSDDEs}
	
	Let $h_\phi$ be given as in \eqref{hFunction} and assume that $h_\phi(z) \neq 0$ for all $z \in \mathbb{C}_+$. Furthermore, let $(X_t)_{t\in \mathbb{R}}$ and $g_\phi$ be defined through \eqref{solution} and \eqref{defOfg}, respectively. The following statements are equivalent:
	\begin{enumerate}[(i)]
	\item\label{p1} $X_t \geq 0$ almost surely for some $t\in \mathbb{R}$.
	\item\label{p2} $X_t \geq 0$ almost surely for all $t\in \mathbb{R}$.
	\item\label{p3} $g_\phi$ is non-negative almost everywhere.
	\item\label{p4} $1/h_\phi$ is completely monotone on $[0,\infty)$, that is, 
	\begin{equation}\label{compMonotone}
	(-1)^n \frac{\dd^n }{\dd x^n}\frac{1}{h_\phi (x)} \geq 0 \qquad \text{for all $x >0$ and $n\in \mathbb{N}_0$.}
	\end{equation}
	\end{enumerate}
	
\end{theorem}

\begin{remark}\label{nonNegRemark}
	Let the setting be as in Theorem~\ref{nonNegativeOfSDDEs}. The notion of almost sure non-negativity is the best possible given that the process $(X_t)_{t\in \mathbb{R}}$ itself is only defined for each fixed $t\in \mathbb{R}$ up to a set of probability zero. However, by \eqref{SDDErelation}, $(X_t)_{t\in \mathbb{R}}$ can always be chosen to have càdlàg sample paths, in which case the property $X_t\geq 0$ will hold across all $t\in \mathbb{R}$ outside a set of probability zero.
\end{remark}

 While Theorem~\ref{nonNegativeOfSDDEs} tells that it is sufficient to show non-negativity of $g_\phi$, the kernel is often not tractable---not even when $\phi$ is rather simple. To give an example where this strategy does indeed work out, note that the solution of \eqref{discDelayEq} with $\lambda >0$ and $\xi = 0$ is the Ornstein--Uhlenbeck process, so $g_\phi (t) = e^{-\lambda t}$ (by \eqref{defOfg} as $h_\phi (z) = z+\lambda$) and, thus, proves that the solution is non-negative. In contrast, as soon as $\xi \neq 0$, $g_\phi$ cannot be explicitly determined since its structure depends on the infinitely many solutions of the equation $z+\lambda - \xi e^{-\tau z}=0$ for $z\in \mathbb{C}$ (see \cite[Lemma~2.1]{GK} for details). The following result, which relies on part~\ref{p4} of Theorem~\ref{nonNegativeOfSDDEs}, provides a sufficient and simple condition on the delay $\phi$ which ensures that the solution is non-negative.

\begin{theorem}\label{NonNegDelay}
		
		Let $h_\phi$ be defined as in \eqref{hFunction} and assume that $h_\phi(z) \neq 0$ for all $z \in \mathbb{C}_+$. Suppose further that $\phi$ is a non-negative measure when restricted to $(0,\infty)$, that is,
		\begin{equation}\label{theoremDelayCond}
		\phi (B\cap (0,\infty)) \geq 0\qquad \text{for all measurable sets $B$.}
		\end{equation}
		Then \ref{p1}--\ref{p4} of Theorem~\ref{nonNegativeOfSDDEs} are satisfied. 

\end{theorem}

\noindent In light of Theorem~\ref{NonNegDelay}, when one is trying to model non-negative processes, it is natural to write \eqref{SDDErelation} as
\begin{equation}\label{nonNegSDDE}
\dd X_t = -\lambda X_t\, \dd t + \int_0^\infty X_{t-s}\, \eta (\dd s)\, \dd t + \dd L_t,\qquad t \in \mathbb{R},
\end{equation}
and then search for non-negative measures $\eta$ on $(0,\infty)$ satisfying $\eta ((0,\infty))<\lambda$ (by Remark~\ref{negDelay} this inequality must hold if $h_\phi (z) \neq 0$ for all $z\in \mathbb{C}_+$). Here we use the convention $\int_0^\infty \coloneqq \int_{(0,\infty)}$.

\begin{remark}\label{diBruno}
	As will, for instance, appear from Example~\ref{CARMA21NonNeg}, the condition \eqref{theoremDelayCond} is not necessary for $(X_t)_{t\in \mathbb{R}}$ to be non-negative. According to Theorem~\ref{nonNegativeOfSDDEs} it is necessary and sufficient that \eqref{compMonotone} is satisfied. The case $n=0$ is always true when $h_\phi (z) \neq 0$ for all $z\in \mathbb{C}_+$, and by relying on Faà di Bruno's formula it can be checked that
	\begin{align}\label{Bruno}
	\begin{aligned}
	\MoveEqLeft (-1)^n \frac{\dd^n }{\dd x^n}\frac{1}{h_\phi (x)}\\
	&= \sum_{\alpha \in \mathbb{N}_0^n\colon \sum_{j=1}^nj\alpha_j = n} 
	\frac{n!\vert \alpha \vert !}{\alpha_1 ! (1!)^{\alpha_1}\cdots \alpha_n! (n!)^{\alpha_n}}  
	\frac{(1-\mathcal{L}^{(1)}_\eta (x))^{\alpha_1}}{h_\phi (x)^{\vert \alpha \vert+1}}\prod_{j=2}^n ((-1)^j\mathcal{L}^{(j)}_\eta (x))^{\alpha_j}
	\end{aligned}
	\end{align}
	when $n\geq 1$, where $\mathcal{L}^{(j)}_\eta$ denotes the $j$th derivative of the Laplace transform $\mathcal{L}_\eta(x)\coloneqq \int_0^\infty e^{-xt}\, \eta (\dd t)$ of $\eta\coloneqq \phi (\: \cdot \: \cap (0,\infty))$. Moreover, $\alpha_j$ refers to the $j$th entry of $\alpha\in \mathbb{N}_0^n$ and $\vert \alpha \vert = \sum_{j=1}^n \alpha_j$. From this expression it is easy to see that is satisfied if $\eta$ is non-negative, since then $\mathcal{L}_\eta$ is completely monotone and, hence, each term of the sum in \eqref{Bruno} will be non-negative. It does, however, also show that $\eta$ cannot be ``too negative'' if the solution $(X_t)_{t\in \mathbb{R}}$ is required to be non-negative; for instance, the restriction for $n=1$ implies in particular that
	\begin{equation*}
	\int_0^\infty t\, \eta (\dd t)\geq -1.
	\end{equation*}
	Unfortunately, the complete set of restrictions implied by \eqref{compMonotone} and \eqref{Bruno} seem to be very difficult to analyze.
\end{remark}

In Figure~\ref{discDelayFigure} we simulate the stationary solution of \eqref{nonNegSDDE} when $\eta =  \xi\delta_\tau$, $\delta_\tau$ being the Dirac measure at $\tau$, with the specific values $\lambda = \tau =1$ and $\xi = 0.2$. For comparison we rerun the simulation with $\xi = -0.8$. Note that existence and uniqueness of the stationary solution is ensured by Example~\ref{discDelay} in both cases. As should be expected, it appears from the simulations that the solution stays non-negative when $\xi = 0.2$, but when $\xi = -0.8$ (where non-negativity is not guaranteed by Theorem~\ref{NonNegDelay}), the solution eventually becomes negative. The latter observation can, for instance, be proved theoretically by checking that
\begin{equation*}
h_\phi (x)^3\frac{\dd^2}{\dd x^2}\frac{1}{h_\phi (x)}\biggr\vert_{x=0} = \xi^2 + 5 \xi +2
\end{equation*}
when $\lambda = \tau = 1$, and hence $1/h_\phi$ is not completely monotone on $[0,\infty)$ if $\xi = -0.8$.

\begin{figure}
	\centering
	\begin{minipage}{0.49\textwidth}
		\begin{tikzpicture}
		\begin{axis}[
		axis lines=middle,
		axis on top = true,
		width=\linewidth, 
		xticklabel style={/pgf/number format/fixed},
		xmax=210,
		ymin=0,
		]
		\addplot[darkgray!50] table [col sep=comma,x=a,y=b,mark=none] {SDDE_pos.csv};
		\end{axis}
		\end{tikzpicture}
	\end{minipage}
	\hfill
	\begin{minipage}{0.49\textwidth}
		\begin{tikzpicture}
		\begin{axis}[
		axis lines=middle, 
		axis on top = true, 
		width=\linewidth, 
		xticklabel style={/pgf/number format/fixed},
		xmax=210,
		]
		\addplot[darkgray!50] table [col sep=comma,x=a,y=b,mark=none] {SDDE_neg.csv};
		\end{axis}
		\end{tikzpicture}
	\end{minipage}
	\caption{Simulations of $X_1,\dots, X_{200}$ from the model \eqref{nonNegSDDE} with $\eta = \xi \delta_1$ when $\xi = 0.2$ (left) and $\xi = -0.8$ (right). In both cases, $\lambda = 1$.}\label{discDelayFigure}
\end{figure}
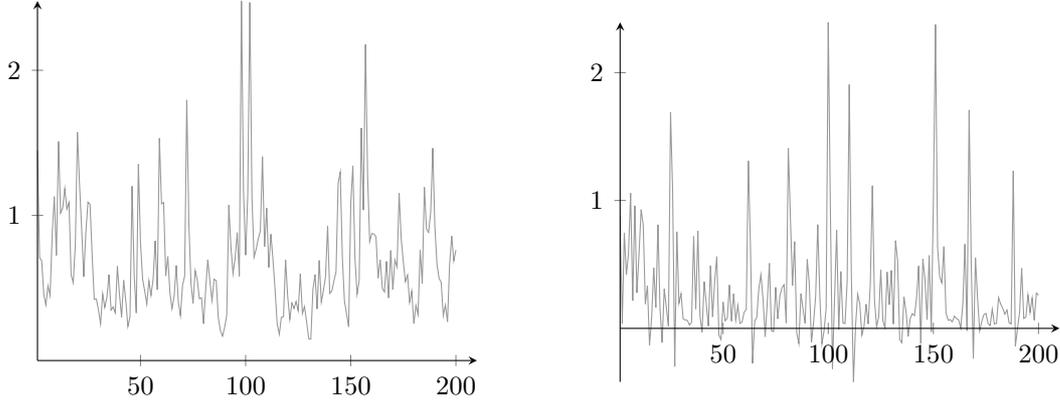

\section{Non-negative CARMA processes}\label{CARMArelations}
Let $P(z) = z^p + a_1z^{p-1}+ \cdots + a_p$ and $Q(z) = b_0 + b_1z + \cdots + b_{q-1}z^{q-1} + z^q$ be two real monic polynomials with $p > q$, and assume that $P$ has no zeroes on $\mathbb{C}_+$ (causality). Define the companion matrix
\begin{equation*}
{\boldsymbol A} = \begin{pmatrix}
0 & 1 & 0 & \cdots & 0 \\
0 & 0 & 1 &\cdots & 0\\
\vdots & \vdots & \ddots & \ddots & \vdots \\
0 & 0 & \cdots & 0 &1 \\
-a_p & -a_{p-1} & \cdots & -a_{2} & -a_{1}
\end{pmatrix},
\end{equation*}
and let ${\boldsymbol b} = (b_0,b_1,\dots, b_{q-1}, 1,0,\dots, 0)^\top$ and ${\boldsymbol e}_p = (0,0,\dots, 0, 1)^\top$ (both being elements of $\mathbb{R}^p$). The causal CARMA($p,q$) process $(Y_t)_{t\in \mathbb{R}}$ driven by the subordinator $(L_t)_{t\in \mathbb{R}}$, and associated to the autoregressive polynomial $P$ and moving average polynomial $Q$, is defined by
\begin{equation}\label{solutionCARMA}
Y_t = \int_{-\infty}^t {\boldsymbol b}^\top e^{{\boldsymbol A}(t-s)}{\boldsymbol e}_p\, \dd L_s,\qquad t \in \mathbb{R}.
\end{equation}
Alternatively, the kernel $g(t) = {\boldsymbol b}^\top e^{{\boldsymbol A}t}{\boldsymbol e}_p$ can be characterized in the frequency domain by the relation
\begin{equation}\label{fourierCARMA}
\int_0^\infty e^{-ity}g(t)\, \dd t = \frac{Q(iy)}{P(iy)},\qquad y \in \mathbb{R}.
\end{equation}
Due to the well-known form ${\boldsymbol X}_t = \int_{-\infty}^t e^{{\boldsymbol A} (t-s)}{\boldsymbol e}_p\, \dd L_s$ of the stationary Ornstein--Uhlenbeck process with drift parameter ${\boldsymbol A}$ and driven by $({\boldsymbol e}_pL_t)_{t\in \mathbb{R}}$ (see, for example, \cite{sato1994recurrence,Sato_OU}), \eqref{solutionCARMA} shows immediately that $(Y_t)_{t\in \mathbb{R}}$ admits the following state-space representation:
\begin{align*}
\dd {\boldsymbol X}_t &= {\boldsymbol A}{\boldsymbol X}_t\, \dd t + {\boldsymbol e}_p\, \dd L_t,\\
Y_t &= {\boldsymbol b}^\top {\boldsymbol X}_t.
\end{align*}
The intuition behind any of the above (equivalent) definitions of the CARMA process is that $(Y_t)_{t\in \mathbb{R}}$ should be the solution of the formal differential equation
\begin{equation}\label{heuristics}
P(D)Y_t = Q(D)D L_t,\qquad t \in \mathbb{R},
\end{equation}
where $D = \frac{\dd}{\dd t}$ denotes derivative with respect to $t$. (For instance, by heuristically computing the Fourier transform of $(Y_t)_{t\in \mathbb{R}}$ from \eqref{heuristics}, one can deduce that $Y_t$ should indeed take the form \eqref{solutionCARMA}.) In general, a wide range of theoretical and applied aspects of CARMA processes have been studied in the literature: for further details, see the survey of \citet{brockwell2014recent} and references therein.

Concerning non-negativity of CARMA processes driven by subordinators, the results are less conclusive; we briefly review existing results here. The simplest CARMA process, obtained with $P(z) = z+\lambda$ for $\lambda >0$ and $Q(z) = 1$, is the Ornstein--Uhlenbeck process, which (as noted in previous sections) is always non-negative when the driving noise $(L_t)_{t\in \mathbb{R}}$ is a subordinator. A particularly appealing property of CARMA processes, in contrast to general solutions of SDDEs, is that the Fourier transform \eqref{fourierCARMA} of the kernel $g$ is rational. To be specific, with $\alpha_1,\dots, \alpha_p$ and $\beta_1,\dots, \beta_q$ being the zeroes of $P$ and $Q$, respectively, \citet{ball1994completely} showed that $g$ is non-negative if 
\begin{align}
\begin{aligned}\label{CARMAsufficiency}
\alpha_1,\dots, \alpha_p,\beta_1,\dots,\beta_q&\in \{z \in \mathbb{C}\, :\, \Re (z)<0,\, \Im (z)=0\}\\
\text{and} \qquad
\sum_{j=1}^k \alpha_j &\geq \sum_{j=1}^k \beta_j\qquad \text{for $k=1,\dots, q$},
\end{aligned}
\end{align}
where the latter condition assumes that the zeroes are ordered such that $\alpha_1\geq \dots \geq \alpha_p$ and $\beta_1\geq \dots \geq \beta_q$. This result was later complemented by \cite{tsai2005note}, which gives two conditions (one necessary and one sufficient) for CARMA($p,0$) or, simply, CAR($p$) processes to be non-negative. Furthermore, it is argued that \eqref{CARMAsufficiency} is both necessary and sufficient to ensure non-negativity of CARMA($2,1$) processes. For many purposes, but not all, CARMA($p,p-1$) processes are the most useful in practice; for instance, CARMA($p,q$) processes have differentiable sample paths when $q<p-1$ (see \cite[Proposition~3.32]{marquardt2007multivariate}). When $q=p-1$ and $Q$ has no zeroes on $\mathbb{C}_+$ (invertibility), we saw in Example~\ref{CARMAprocesses} that the corresponding causal and invertible CARMA process $(Y_t)_{t\in \mathbb{R}}$ can be identified as the unique solution of an SDDE of the form \eqref{nonNegSDDE} with $\eta (\dd t) = f(t)\, \dd t$ and $f\colon [0,\infty) \to \mathbb{R}$ being characterized by
\begin{equation}\label{fDef}
\int_0^\infty e^{-ity}f(t)\, \dd t = \frac{R(iy)}{Q(iy)},\qquad y \in \mathbb{R}.
\end{equation}
Here $R(z)\coloneqq (z+\lambda)Q(z) - P(z)$ and $\lambda \in \mathbb{R}$ is uniquely determined by the condition $\deg (R) < q$; it is explicitly given by 
\begin{equation*}
\lambda = a_1 - b_{p-2} = -\alpha_p+\sum_{j=1}^{p-1}(\beta_j-\alpha_j).
\end{equation*}
Note that, in case the zeroes $\beta_1,\dots, \beta_q$ of $Q$ are distinct, Cauchy's residue theorem implies that
\begin{equation}\label{explicitf}
f(t) = -\sum_{j=1}^q \frac{P (\beta_j)}{Q'(\beta_j)}e^{\beta_j t},\qquad t \geq 0,
\end{equation}
$Q'$ being the derivative of $Q$ (see also \cite[Remark~5]{brockwell2014recent}). In view of Theorem~\ref{NonNegDelay}, it follows that $(Y_t)_{t\in \mathbb{R}}$ is non-negative if $f$ is so. To put it differently, while it is necessary and sufficient that $Q/P$ is completely monotone on $[0,\infty)$ for $(Y_t)_{t\in \mathbb{R}}$ to be non-negative, it is sufficient to verify complete monotonicity of $R/Q$, a rational function where both the numerator and denominator are of lower order than the original one. We state this finding in the following result.

\begin{theorem}\label{CARMAasSDDEs}
	Let $(Y_t)_{t\in \mathbb{R}}$ be a causal and invertible CARMA($p,p-1$) process associated to the pair $(P,Q)$, and let $R$ be given as above. Then $(Y_t)_{t\in \mathbb{R}}$ is non-negative if $R/Q$ is completely monotone on $[0,\infty)$ or, equivalently, if the function $f$ characterized by \eqref{fDef} is non-negative.
\end{theorem}

\begin{remark}\label{CARMApq}
	In situations where $q<p-1$, Theorem~\ref{CARMAasSDDEs} can still be useful for obtaining non-negative CARMA($p,q$) processes as it can be applied in conjunction with the results of \cite{tsai2005note}. To be specific, suppose that
	\begin{enumerate}[(i)]
	 \item\label{point1} Theorem~\ref{CARMAasSDDEs} applies to the CARMA($q+1,q$) process associated with a given the pair $(P_1,Q)$, and
	 \item\label{point2} one of the sufficient conditions of \cite[Theorem~1]{tsai2005note} applies to the CAR($p-q-1$) process associated with a given autoregressive polynomial $P_2$.
	 \end{enumerate}
 	Then the CARMA($p,q$) process associated with the pair $(P_1\times P_2,Q)$ is non-negative ($\times$ denoting multiplication). This follows immediately from the fact that its kernel is the convolution between the kernels of the two moving averages associated to \ref{point1} and \ref{point2}.
\end{remark}

\begin{example}\label{CARMA21NonNeg}
	Suppose that $(Y_t)_{t\in\mathbb{R}}$ is a causal and invertible CARMA($2,1$) process, so that its polynomials $P$ and $Q$ are representable as $P(z) = (z-\alpha)(z-\beta)$ and $Q(z) = z- \gamma$ for suitable $\alpha,\beta \in \{z \in \mathbb{C}\, :\, \Re(z) <0\}$ and $\gamma \in (-\infty,0)$. For $(Y_t)_{t\in \mathbb{R}}$ to be non-negative it follows from \eqref{CARMAsufficiency} that it is necessary and sufficient that 
	\begin{equation}\label{CARMA21their}
	\Im (\alpha) = \Im (\beta)=0\qquad \text{and}\qquad \gamma \leq \max\{\alpha,\beta \}.
	\end{equation}
	 On the other hand, from \eqref{explicitf} we have $f(t) = - P(\gamma)e^{\gamma t}$, and so one would need to require that $P(\gamma) = (\gamma-\alpha)(\gamma-\beta)\leq 0$ in order to apply Theorem~\ref{CARMAasSDDEs}. Imposing this condition is equivalent to assuming that 
	\begin{equation}\label{CARMA21our}
	\Im (\alpha) = \Im (\beta) = 0\qquad  \text{and} \qquad \min \{\alpha, \beta\}\leq \gamma \leq \max \{\alpha, \beta\}.
	\end{equation}
	While the first restriction in \eqref{CARMA21our} is necessary, the second is not. In general, we see that \eqref{CARMA21their} and \eqref{CARMA21our} differ the most when $\vert \alpha- \beta\vert$ is small. In Figure~\ref{CARMA21figure} we have marked the feasible area for $\gamma$ implied by \eqref{CARMA21their} and \eqref{CARMA21our}, respectively, for a given polynomial $P$ with zeroes $\alpha, \beta \in (-\infty ,0)$.
	
	\begin{figure}
		\centering
		\begin{minipage}{0.65\textwidth}
			\begin{tikzpicture}
			\begin{axis}[
			axis lines=middle,
			width=\linewidth, 
			xtick = {0.6715,1.9667},
			xticklabels = {$\alpha$, $\beta$},
			ytick = {0},
			xticklabel style={/pgf/number format/fixed},
			every axis plot/.append style={thick},
			axis line style = thick,
			ymax = 0.55,
			xmin = -0.2,
			xmax = 2.32,
			]
			\addplot[darkgray!50] table [col sep=comma,x=a,y=b,mark=none] {CARMA21p.csv};
			\draw[very thick, red!50] (axis cs: 0.6715,0)--(axis cs: 1.9667,0);
			\draw[very thick, red!50, loosely dashed] (axis cs: -1.5,0)--(axis cs: 0.6715,0);
			\end{axis}
			\end{tikzpicture}
		\end{minipage}
		\caption{An autoregressive polynomial $P$ with zeroes $\alpha, \beta \in (-\infty,0)$ (gray) together with the feasible values for the zero $\gamma \in (-0,\infty)$ of the corresponding moving average polynomial $Q$ implied by \eqref{CARMA21their} (red: solid and dashed) and \eqref{CARMA21our} (red: solid).}\label{CARMA21figure}
	\end{figure}
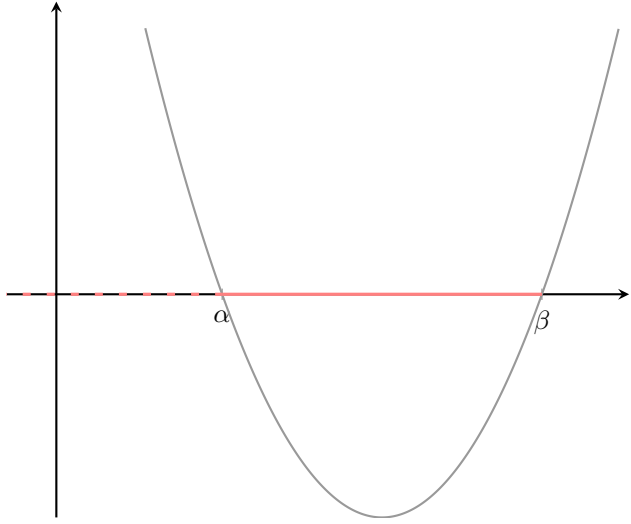
\end{example}

While the condition of Theorem~\ref{CARMAasSDDEs} falls short in the context of CARMA($2,1$) processes, Corollary~\ref{CARMA32processes} below shows that it extends the class of CARMA($p,p-1$) processes which is known to be non-negative as soon as $p\geq 3$. Specifically, this result gives sufficient conditions ensuring that CARMA($3,2$) processes are non-negative in situations where those of \cite{ball1994completely,tsai2005note} do not (and the latter are, to the best our knowledge, the only available easy-to-check conditions in the literature). One could aim for a similar result for a general $p \geq 3$, but this would involve more complicated expressions.

\begin{corollary}\label{CARMA32processes}
	Let $(Y_t)_{t\in \mathbb{R}}$ be a causal and invertible CARMA($3,2$) process with polynomials $P$ and $Q$, and let $\beta_1$ and $\beta_2$, $\Re(\beta_1)\geq \Re(\beta_2)$, denote the zeroes of $Q$. Then the function $f$ characterized by \eqref{fDef} is non-negative if and only if $\Im (\beta_1) =\Im (\beta_2) = 0$ and one of the following conditions holds:
	\begin{enumerate}[(i)]
		\item\label{condA} $\beta_1 >\beta_2$ and $P(\beta_1) \leq \min \{P(\beta_2),0\}$.
		\item\label{condB} $\beta_1 = \beta_2$ and $\max\{P(\beta_1),P'(\beta_1)\} \leq 0$.
	\end{enumerate}
	In particular, if either \ref{condA} or \ref{condB} is satisfied, then $(Y_t)_{t\in \mathbb{R}}$ is non-negative.
\end{corollary}

\begin{example}\label{CARMA32NonNeg}
	Let $(Y_t)_{t\in \mathbb{R}}$ be the causal and invertible CARMA($3,2$) associated to real polynomials $P$ and $Q$ representable as
	\begin{equation*}
	P(z) = (z-\alpha_1)(z-\alpha_2)(z-\alpha_3)\quad \text{and} \quad Q(z) = (z-\beta_1)(z-\beta_2),\qquad z \in \mathbb{C},
	\end{equation*}
	for some $\alpha_1,\alpha_2,\alpha_3 \in \{z\in \mathbb{C}\, :\, \Re (z) <0\}$ and $\beta_1,\beta_2\in (-\infty,0)$. If one seeks to apply the result of \cite{ball1994completely,tsai2005note} stated in \eqref{CARMAsufficiency}, it must be required that
	\begin{equation}\label{CARMA32suff1}
	\Im (\alpha_j) = 0,\qquad \beta_1\leq \alpha_1,\qquad \text{and}\qquad \beta_1 + \beta_2 \leq \alpha_1 +\alpha_2.
	\end{equation}
	(Here, as in \eqref{CARMAsufficiency}, the zeroes are ordered such that $\alpha_1 \geq \alpha_2\geq \alpha_3$ and $\beta_1 \geq \beta_2$.) Comparing \eqref{CARMA32suff1} to the conditions obtained in Corollary~\ref{CARMA32processes}, it follows that they are quite different. In particular, Corollary~\ref{CARMA32processes} does not require that $\alpha_1$, $\alpha_2$, and $\alpha_3$ are real numbers, and even if this is the case, the conditions of this result are sometimes met while those of \eqref{CARMA32suff1} are not, and vice versa. To give a simple example of the former case, suppose that $\alpha_1 >\alpha_2 = \alpha_3$ and $\beta_1 = \beta_2$. If $\beta_1$ is larger than $\alpha_2$, but smaller than the local extremum $x^\ast$ of $P$ located in the interval $(\alpha_2,\alpha_1)$, it follows that both $P(\beta_1)\leq 0$ and $P'(\beta_1)\leq 0$. It is easy to check that $x^\ast = (2\alpha_1+ \alpha_2)/3$, so we conclude that the conditions of Corollary~\ref{CARMA32processes} are satisfied while those of \eqref{CARMA32suff1} are not whenever
	\begin{equation*}
	\frac{\alpha_1 + \alpha_2}{2}< \beta_1 \leq \frac{\alpha_1 + \alpha_2}{2} + \frac{\alpha_1 - \alpha_2}{6}.
	\end{equation*}
	 On the other hand, a simple example where \eqref{CARMA32suff1} applies, but Corollary~\ref{CARMA32processes} does not, is when $\alpha_3 < \beta_1<\alpha_2$, since this implies $P(\beta_1)>0$.
\end{example}

\section{Multivariate extensions}\label{multivariate}
In this section we extend the results of Section~\ref{SDDEnonNeg} to multivariate SDDEs or, in short, MSDDEs which were studied in \cite{basse2019multivariate}. To be specific, for a given $d \geq 2$ we let ${\boldsymbol L}_t = (L^1_t,\dots, L^d_t)^\top$, $t\in \mathbb{R}$, be a $d$-dimensional subordinator with $\mathbb{E}[\lVert {\boldsymbol L}_1 \rVert]<\infty$ and ${\boldsymbol\phi} = (\phi_{jk})$ a $d\times d$ matrix-valued signed measure on $[0,\infty)$ with
\begin{equation*}
\int_0^\infty t^2\, \vert \phi_{jk}\vert (\dd t)<\infty,\qquad j,k=1,\dots, d.
\end{equation*}
A stochastic process ${\boldsymbol X}_t = (X^1_t, \dots, X^d_t)^\top$, $t\in \mathbb{R}$, is a solution of the corresponding MSDDE if it is stationary, has finite first moments (that is, $\mathbb{E}[\lVert {\boldsymbol X}_0 \rVert]<\infty$), and satisfies
\begin{equation}\label{MSDDE}
X^j_t - X^j_s = \sum_{k=1}^d \int_s^t\int_{[0,\infty)}X^k_{u-v}\, \phi_{jk} (\dd v)\, \dd u + L^j_t -L^j_s,\qquad s <t,
\end{equation}
for $j=1,\dots, d$. In line with \eqref{NonNegDelay} we use the decomposition ${\boldsymbol\phi} = -{\boldsymbol\lambda} \delta_0 + {\boldsymbol\eta}$, where ${\boldsymbol\lambda} \coloneqq - {\boldsymbol\phi} (\{0\})$ and ${\boldsymbol\eta} \coloneqq {\boldsymbol\phi} (\: \cdot \: \cap (0,\infty))$. If we define the convolution $(\mu\ast f) (t) = (f \ast \mu) (t)  \coloneqq \int f(t-s)\, \mu (\dd s)$ between a signed measure and a real-valued function $f$ (assuming that the integral exists) and extend the definition to matrix-valued quantities by the usual rules of matrix multiplication, \eqref{MSDDE} can be written compactly as
\begin{equation*}
\dd {\boldsymbol X}_t = - {\boldsymbol \lambda} {\boldsymbol X}_t \,\dd t + ({\boldsymbol \eta}\ast {\boldsymbol X}) (t)\, \dd t + \dd {\boldsymbol L}_t,\qquad t \in \mathbb{R}.
\end{equation*}
Among stationary processes which can be identified as solutions of \eqref{MSDDE} are the multivariate Ornstein--Uhlenbeck process (${\boldsymbol \eta} \equiv {\boldsymbol 0}$) and, more generally, multivariate CARMA processes \cite{marquardt2007multivariate} (${\boldsymbol \eta}(\dd t) = {\boldsymbol f}(t)\, \dd t$ with ${\boldsymbol f}$ being of exponential type as in~\eqref{CARMAdelay}). For further details, we refer to \cite{basse2019multivariate}. In general, solutions of MSDDEs are related to a function which is similar to \eqref{hFunction}, namely
\begin{equation}\label{hFunctionMulti}
{\boldsymbol h}_{\phi} (z) = z{\boldsymbol I}_d - \int_{[0,\infty)}e^{-zt}\, {\boldsymbol \phi} (\dd t),\qquad z\in \mathbb{C}_+,
\end{equation}
${\boldsymbol I}_d$ being the $d\times d$ identity matrix. In particular, it was shown in \cite{basse2019multivariate} that if $\det ({\boldsymbol h}_{\phi}(z))\neq 0$ for all $z\in \mathbb{C}_+$, the unique solution $({\boldsymbol X}_t)_{t\in \mathbb{R}}$ is given by
\begin{equation}\label{solutionMulti}
{\boldsymbol X}_t = \int_{-\infty}^t {\boldsymbol g}_\phi (t-s)\, \dd {\boldsymbol L}_s, \qquad t \in \mathbb{R},
\end{equation}
where ${\boldsymbol g}_\phi\colon [0,\infty)\to \mathbb{R}^{d\times d}$ is characterized by
\begin{equation}\label{gCharacterize}
\int_0^\infty e^{-ity}{\boldsymbol g}_\phi (t)\, \dd t = {\boldsymbol h}_\phi (iy)^{-1},\qquad y \in \mathbb{R}.
\end{equation}
In Theorem~\ref{NonNegDelayMulti}, which is a generalization of Theorem~\ref{NonNegDelay}, we give sufficient conditions for the solution $({\boldsymbol X}_t)_{t\in \mathbb{R}}$ to have non-negative entries. Before formulating this result we recall the definition of the so-called $M$-matrices: ${\boldsymbol A}\in \mathbb{R}^{d\times d}$ is called an $M$-matrix if it can be expressed as ${\boldsymbol A} = \alpha {\boldsymbol I}_d - {\boldsymbol B}$ for some $\alpha \in [0,\infty)$ and ${\boldsymbol B}\in \mathbb{R}^{d\times d}$ with non-negative entries and spectral radius $\rho ({\boldsymbol B})\leq \alpha$. The main point in introducing these is their relation to entrywise non-negativity of the matrix exponential $e^{-{\boldsymbol A}t}$ for all $t\geq 0$ (see Lemma~\ref{NonNegMatExp}). For further details and other characterizations of $M$-matrices, we refer to \cite{berman1994nonnegative,plemmons1977m} and references therein.

\begin{theorem}\label{NonNegDelayMulti}
	Let ${\boldsymbol h}_\phi$ be defined as in \eqref{hFunctionMulti} and assume that $\det ({\boldsymbol h}_\phi(z)) \neq 0$ for all $z \in \mathbb{C}_+$. Suppose further that all the entries of ${\boldsymbol \eta}$ are non-negative measures and that ${\boldsymbol \lambda}$ is an $M$-matrix. Then the entries of the unique solution $({\boldsymbol X}_t)_{t \in \mathbb{R}}$ to \eqref{SDDErelation}, which is given by \eqref{solution}, are non-negative in the sense that $X^j_t \geq 0$ for $j=1,\dots, d$ almost surely for each fixed $t\in \mathbb{R}$. 
\end{theorem}

\begin{remark}
	When comparing Theorem~\ref{NonNegDelayMulti} to Theorem~\ref{NonNegDelay}, it seems that we have to impose an additional assumption of ${\boldsymbol \lambda}$ being an $M$-matrix, which was not needed in the univariate case. However, requiring that the one-dimensional quantity $\lambda = -\phi (\{0\})$ is an $M$-matrix would naturally be interpreted as a non-negativity constraint, but this is automatically satisfied under the stated assumptions on $h_\phi$, cf.~Remark~\ref{nonNegRemark}.
\end{remark}

In the same way as we relied on the findings of Section~\ref{SDDEnonNeg} to obtain conditions for a CARMA process to be non-negative in Section~\ref{CARMArelations}, Theorem~\ref{NonNegDelayMulti} can be used in conjunction with the relation between MSDDEs and multivariate CARMA processes outlined in \cite{basse2019multivariate} to obtain similar conditions in the multivariate setting.

\section{Proofs}\label{proofs}
\begin{proof}[Proof of Theorem~\ref{nonNegativeOfSDDEs}]
	The equivalence between \ref{p1} and \ref{p2} is an immediate consequence of stationarity of $(X_t)_{t\in \mathbb{R}}$. The fact that \ref{p3} implies \ref{p2} is rather obvious---this is also readily seen by representing the subordinator $(L_t)_{t\in \mathbb{R}}$ as a cumulative sum of its positive jumps plus a non-negative drift. To see that \ref{p2} implies \ref{p3}, note by \cite{rosSpec} that $X_t$ is infinitely divisible with a Lévy measure $\nu_X$ given by
	\begin{equation*}
	\nu_X (B)\coloneqq \bigl(\text{Leb}\times \nu\bigr)\bigl(\{(s,x)\in [0,\infty)\times \mathbb{R}\, :\, xg_\phi (s)\in B\setminus\{0\}\}\bigr),
	\end{equation*}
	where $\text{Leb}(\: \cdot\:)$ refers to the Lebesgue measure. Since $X_t\geq 0$ almost surely, $\nu_X$ is concentrated on $(0,\infty)$, and hence we deduce that
	\begin{equation}\label{zeroLebesgue}
	0 = \nu_X ((-\infty,0)) \geq \text{Leb} (\{s\in [0,\infty)\, :\, g_\phi (s) <0\})\nu ((0,\infty)).
	\end{equation}
	Since $(L_t)_{t\in \mathbb{R}}$ is a subordinator and $\nu$ is non-zero, $\nu ((0,\infty))>0$. By combining this observation with \eqref{zeroLebesgue} we conclude that $g_\phi$ is non-negative almost everywhere. In view of Bernstein's theorem on monotone functions \cite{bernstein1929fonctions}, which states that $g_\phi$ is non-negative if and only if its Laplace transform is completely monotone on $[0,\infty)$, we have that \ref{p3} holds if and only if \ref{p4} does, and this completes the proof.
	
\end{proof}

\begin{proof}[Proof of Theorem~\ref{NonNegDelay}]
	This follows immediately from the discussion in Remark~\ref{diBruno} or by observing that $x\mapsto h_\phi (x)$ is a Bernstein function, and hence its reciprocal $x\mapsto 1/h_\phi (x)$ must be completely monotone (see \cite[Theorem~5.4]{berg2008stieltjes}).
\end{proof}

\begin{proof}[Proof of Corollary~\ref{CARMA32processes}] Since $f$ takes the form \eqref{fDef} and $\deg (Q) = 2$, it follows from \cite[Remark~2]{tsai2005note} that, necessarily, $\Im (\beta_1) = \Im (\beta_2) = 0$ (as we also discussed in relation to \eqref{CARMAsufficiency}). If $\beta_1 > \beta_2$, we may use \eqref{explicitf} and represent $f$ explicitly as
\begin{equation*}
f(t) = \frac{1}{\beta_1-\beta_2}\bigl(P(\beta_2)e^{\beta_2 t} - P(\beta_1)e^{\beta_1 t} \bigr) ,\qquad t \geq 0.
\end{equation*}
Here we have also used the fact that $Q'(\beta_1)  = - Q'(\beta_2) = \beta_1  - \beta_2$. By considering $t=0$ and $t\to \infty$ we conclude that $f (t) \geq 0$ for all $t$ if and only if both $P(\beta_1)\leq P(\beta_2)$ and $P(\beta_1)\leq 0$, and this thus establishes \ref{condA}. If $\beta_1 = \beta_2\eqqcolon \beta$, we observe initially that 
\begin{align*}
R(z) &=  (z+a_1+2\beta)(z-\beta)^2 - P(z)\\
&= 2\beta^3 + a_1\beta^2 - a_3 -(3 \beta^2+ 2\beta a_1 - a_2) z\\
&= P'(\beta)\beta-P(\beta) -P'(\beta)z.
\end{align*}
Since the function having Fourier transform $y \mapsto (iy-\beta)^{-2}$ is $t\mapsto te^{\beta t}$, it follows from \eqref{fDef} and the above expression for $R$ that
\begin{equation*}
f(t) = (P'(\beta)\beta - P(\beta))te^{\beta t} - P'(\beta)(\beta t e^{\beta t}+e^{\beta t})= - (P(\beta)t + P'(\beta))e^{\beta t}.
\end{equation*}
By considering $t=0$ and $t\to \infty$ once again we deduce that $f$ is non-negative on $[0,\infty)$ if and only if $P(\beta)\leq 0$ and $P'(\beta)\leq 0$, which concludes the proof.	
\end{proof}

\noindent Before turning to the proof of Theorem~\ref{NonNegDelayMulti} we will need a couple of auxiliary lemmas. In relation to the formulation of the first result we recall the notation
\begin{equation*}
({\boldsymbol \mu}\ast {\boldsymbol f})_{jk} = \sum_l \int f_{lk}(\: \cdot \: - t)\, \mu_{jl} (\dd t)
\end{equation*}
for a matrix-valued function ${\boldsymbol f} = (f_{jk})$ and a matrix-valued signed measure ${\boldsymbol \mu} = (\mu_{jk})$ of suitable dimensions and such that the involved integrals are well-defined.
\begin{lemma}\label{induction}
	Suppose that $\det ({\boldsymbol h}_\phi (z)) \neq 0$ for all $z\in \mathbb{C}_+$ and let ${\boldsymbol g}_\phi \colon [0,\infty)\to \mathbb{R}^{d\times d}$ be characterized by \eqref{gCharacterize}. Then it holds that
	\begin{equation*}
	{\boldsymbol g}_\phi (t) = {\boldsymbol g}_\phi (t-s) {\boldsymbol g}_\phi (s) + \int_s^t {\boldsymbol g}_\phi (t-u)
	( {\boldsymbol \phi} \ast ({\boldsymbol g}_\phi \mathds{1}_{[0,s]}) ) (u)\, \dd u
	\end{equation*}
	for arbitrary $s,t\in [0,\infty)$ with $s<t$.
\end{lemma}

\begin{proof}
	To lighten notation, and in line with Remark~\ref{diBruno}, we denote the Laplace transform by $\mathcal{L}$. Specifically, for a given signed measure $\mu$ and an integrable function $f\colon [0,\infty)\to \mathbb{R}$, set $\mathcal{L}[\mu](z) = \mathcal{L}[\mu (\dd t)](z) \coloneqq \int_{[0,\infty)}e^{-zt}\,  \mu (\dd t)$ for $z\in \mathbb{C}_+$ and $\mathcal{L}[f] \coloneqq \mathcal{L}[f(t)\, \dd t]$. From \cite[Proposition~5.1]{basse2019multivariate} it follows that
	\begin{equation}\label{step1}
	{\boldsymbol g}_\phi (t)  = {\boldsymbol g}_\phi (s) + \int_s^t ({\boldsymbol g}_\phi \ast {\boldsymbol \phi}) (u)\, \dd u,\qquad t>s\geq 0,
	\end{equation}
	where $({\boldsymbol g}_\phi \ast {\boldsymbol \phi})_{jk} \coloneqq \sum_{l=1}^d \int_{[0,\infty)} g_{\phi ,jl} (\:\cdot\: -t)\, \phi_{lk} (\dd t)$. Since $\mathcal{L}[{\boldsymbol g}_\phi](z) = {\boldsymbol h}_\phi (z)^{-1}$ by \eqref{gCharacterize}, it is easy to check that $\mathcal{L}[{\boldsymbol g}_\phi]$ and $\mathcal{L}[{\boldsymbol \phi}]$ commute, and hence
	\begin{equation}\label{step2}
	{\boldsymbol g}_\phi \ast {\boldsymbol \phi} = {\boldsymbol \phi}\ast {\boldsymbol g}_\phi.
	\end{equation}
	For a fixed $s\in [0,\infty)$, it thus follows from \eqref{step1} and \eqref{step2} that
	\begin{align*}
	\MoveEqLeft\mathcal{L}[{\boldsymbol g}_\phi \mathds{1}_{(s,\infty)}](z) \\
	&= \mathcal{L}[\mathds{1}_{(s,\infty)}](z){\boldsymbol g}_\phi (s)
	+ \mathcal{L}\Bigl[\int_0^{\: \cdot \:} ({\boldsymbol \phi}\ast {\boldsymbol g}_\phi)(u) \mathds{1}_{(s,\infty)}(u)\, \dd u \Bigr] (z)\\
	&=\frac{1}{z}\bigl(\mathcal{L}[\delta_s](z) {\boldsymbol g}_\phi (s) + \mathcal{L}[({\boldsymbol \phi}\ast ({\boldsymbol g}_\phi\mathds{1}_{[0,s]}))\mathds{1}_{(s,\infty)}](z) + \mathcal{L}[{\boldsymbol \phi}](z) \mathcal{L}[{\boldsymbol g}_\phi \mathds{1}_{(s,\infty)}](z)\bigr)
	\end{align*}
	for $z\in \mathbb{C}$ with $\Re (z)>0$. After rearranging terms we obtain
	\begin{equation}\label{intermedEq}
	{\boldsymbol h}_\phi (z) \mathcal{L}[{\boldsymbol g}_\phi \mathds{1}_{(s,\infty)}](z) = \mathcal{L}[\delta_s](z) {\boldsymbol g}_\phi (s) + \mathcal{L}[({\boldsymbol \phi}\ast ({\boldsymbol g}_\phi\mathds{1}_{[0,s]}))\mathds{1}_{(s,\infty)}](z).
	\end{equation}
	By multiplying $\mathcal{L}[{\boldsymbol g}_\phi](z)$ from the left on both sides of \eqref{intermedEq} and using uniqueness of the Laplace transform,
	this shows that
	\begin{align*}
	{\boldsymbol g}_\phi (t) &= ({\boldsymbol g}_\phi \ast \delta_s)(t) {\boldsymbol g}_\phi (s)
	+ ({\boldsymbol g}_\phi \ast ({\boldsymbol \phi}\ast ({\boldsymbol g}_\phi \mathds{1}_{[0,s]})))(t) \\
	&={\boldsymbol g}_\phi (t-s){\boldsymbol g}_\phi (s) + \int_s^t {\boldsymbol g}_\phi (t-u) ({\boldsymbol \phi}\ast ({\boldsymbol g}_\phi \mathds{1}_{[0,s]}))(u)\, \dd u
	\end{align*}
	for $t>s$, and this completes the proof.
\end{proof}

\noindent The following lemma, which will be used in the proof of Theorem~\ref{NonNegDelayMulti} as well, presents a key property of $M$-matrices. While the result is well-known, we have not been able to find a proper reference, and hence we include a small proof of the statement here. 
\begin{lemma}\label{NonNegMatExp}
	Let ${\boldsymbol A}\in \mathbb{R}^{d \times d}$. If ${\boldsymbol A}$ is an $M$-matrix, then $e^{-{\boldsymbol A} t}$ has non-negative entries for all $t\geq 0$. Conversely, if $e^{-{\boldsymbol A}t}$ has non-negative entries for all $t\geq 0$ and the spectrum $\sigma ({\boldsymbol A})$ of ${\boldsymbol A}$ belongs to $\{z\in \mathbb{C}\, :\, \Re (z) >0\}$, then ${\boldsymbol A}$ is an invertible $M$-matrix.
\end{lemma}
\begin{proof}
	If ${\boldsymbol A}$ is an $M$-matrix, we have that $e^{-{\boldsymbol A}t} = e^{-\alpha t}e^{{\boldsymbol B}t}$ for some $\alpha \in [0,\infty)$ and ${\boldsymbol B}\in \mathbb{R}^{d\times d}$ with non-negative entries. Since $e^{{\boldsymbol B}t}$ has non-negative entries, we conclude that the same holds for $e^{-{\boldsymbol A}t}$. If instead the entries of $e^{-{\boldsymbol A}t}$ are non-negative and $\sigma ({\boldsymbol A})\subseteq \{z \in \mathbb{C}\, :\, \Re (z)>0\}$, it follows that the entries of
	\begin{equation*}
	\int_0^\infty e^{-xt}e^{-{\boldsymbol A}t}\, \dd t = (x{\boldsymbol I}_d + {\boldsymbol A})^{-1}
	\end{equation*}
	are non-negative for all $x\in [0,\infty)$ as well. By \cite[Theorem~2]{plemmons1977m} this implies that ${\boldsymbol A}$ is an invertible $M$-matrix.
\end{proof}

\begin{proof}[Proof of Theorem~\ref{NonNegDelayMulti}]
	For $\varepsilon \in (0,\infty)$ define the measure ${\boldsymbol \phi}_\varepsilon \coloneqq -{\boldsymbol \lambda}\delta_0 + {\boldsymbol \eta}(\: \cdot \: \cap (\varepsilon,\infty))$. We start by observing that, when $\varepsilon$ is sufficiently small, 
	\begin{equation}\label{smallEps}
	\det ({\boldsymbol h}_{\phi_\varepsilon} (z)) \neq 0\qquad \text{for all $z\in \mathbb{C}_+$}.
	\end{equation}
	 To see this suppose, for the sake contradiction, there exist sequences $(\varepsilon_n)_{n \geq 1}\subseteq (0,\infty)$ and $(z_n)_{n \geq 1}\subseteq \mathbb{C}_+$ such that $\det ({\boldsymbol h}_{\phi_{\varepsilon_n}} (z_n)) = 0$ for all $n \geq 1$ and $\varepsilon_n\to 0$ as $n \to \infty$. Note that the length of any each entry of $\mathcal{L}[{\boldsymbol \phi}_{\varepsilon}](z)$ is bounded by the constant
	\begin{equation*}
	\max_{j,k=1,\dots, d} \vert \phi_{jk}\vert ([0,\infty)),
	\end{equation*}
	for all $\varepsilon$ and $z$, and hence $\inf_\varepsilon \vert \det ({\boldsymbol h}_{\phi_{\varepsilon}}(z))\vert \sim \vert z\vert^d$ as $\vert z \vert \to \infty$ by the Leibniz formula (the notation $\sim$ means that the ratio tends to one). This shows in particular that the sequence $(z_n)_{n\geq 1}$ must be bounded, and so it has a subsequence $(z_{n_k})_{k\geq 1}$ which converges to a point $z^\ast\in \mathbb{C}_+$. However, this would imply that
	\begin{equation*}
	 \det ({\boldsymbol h}_\phi (z^\ast)) = \lim_{k\to \infty} \det ({\boldsymbol h}_{\phi_{\varepsilon_{n_k}}} (z_{n_k})) = 0,
	\end{equation*}
	which contradicts the original assumption that $\det ({\boldsymbol h}_\phi (z))\neq 0$ for all $z\in \mathbb{C}_+$. Consequently, the condition \eqref{smallEps} is satisfied as long as $\varepsilon$ is smaller than a certain threshold, say, $\varepsilon^\ast$. Thus, for $\varepsilon \leq \varepsilon^\ast$ we can define the corresponding function ${\boldsymbol g}_{\phi_\varepsilon}\colon [0,\infty)\to \mathbb{R}^{d\times d}$ through \eqref{gCharacterize} (with ${\boldsymbol\phi}$ replaced by ${\boldsymbol\phi}_{\varepsilon}$). By \eqref{step1} and \eqref{step2}, 
	\begin{align*}
	{\boldsymbol g}_{\phi_\varepsilon}(t) &= {\boldsymbol I}_d + \int_0^t ({\boldsymbol \phi}_\varepsilon\ast{\boldsymbol g}_{\phi_\varepsilon})(s)\, \dd s 
	\\
	&= {\boldsymbol I}_d - {\boldsymbol \lambda}\int_0^t {\boldsymbol g}_{\phi_\varepsilon} (s)\, \dd s
	\end{align*}
	for $t \in [0,\varepsilon]$. By uniqueness of solutions of such differential equation it follows that for ${\boldsymbol g}_{\phi_\varepsilon} (t) = e^{-{\boldsymbol \lambda}t}$ for $t\in [0,\varepsilon]$, in which case we can rely on Lemma~\ref{NonNegMatExp} to deduce that its entries are non-negative. Now, given that ${\boldsymbol g}_{\phi_\varepsilon} (t)$ has non-negative entries for $t\in [0,k\varepsilon]$ for some positive integer $k \geq 1$, we can apply Lemma~\ref{induction} with $s= k\varepsilon$ to deduce that this remains true when $t\in (k\varepsilon,(k+1)\varepsilon]$. Consequently, it follows by induction that ${\boldsymbol g}_{\phi_\varepsilon}(t)$ has non-negative entries for all $t\in [0,\infty)$. 
	
	Next, for $\varepsilon \in [0,\varepsilon^\ast]$ define the Fourier transform
	\begin{equation*}
	\mathcal{F}_{\phi_\varepsilon, jk} (y) \coloneqq ({\boldsymbol h}_{\phi_\varepsilon}(iy)^{-1})_{jk},\qquad y \in \mathbb{R},
	\end{equation*}
	of the $(j,k)$-th entry $g_{\phi_\varepsilon, jk}$ of ${\boldsymbol g}_{\phi_\varepsilon}$. By following the same arguments as in the proof of \cite[Proposition~5.1]{basse2019multivariate} (in particular, relying on Cramer's rule), one can show that
	\begin{equation}\label{domConv}
	\sup_{0\leq \varepsilon \leq \varepsilon^\ast}\vert \mathcal{F}_{\phi_\varepsilon, jk} (y)\vert \leq C (1\wedge \vert y \vert^{-1}),\qquad y \in \mathbb{R}.
	\end{equation}
	for a suitable constant $C\in (0,\infty)$. Since $\mathcal{F}_{\phi_\varepsilon, jk}$ converges pointwise to $\mathcal{F}_{\phi_0,jk}$ (the Fourier transform of the $(j,k)$-th entry $g_{\phi, jk}$ of ${\boldsymbol g}_\phi$) as $\varepsilon \to 0$, it follows by the Plancherel theorem and \eqref{domConv} that
	\begin{equation*}
	\int_0^\infty (g_{\phi_\varepsilon, jk}(t)-g_{\phi, jk}(t))^2\, \dd t = \int_\mathbb{R} \vert\mathcal{F}_{\phi_\varepsilon, jk}(y)-\mathcal{F}_{\phi_0,jk}(y) \vert^2\, \dd y \to 0,\qquad \varepsilon \to 0.
	\end{equation*}
	From this we establish that $g_{\phi_{\varepsilon_n},jk}\to g_{\phi,jk}$ almost everywhere for a suitable sequence $(\varepsilon_n)_{n \geq 1}\subseteq (0,\varepsilon^\ast]$ with $\varepsilon_n \to 0$ as $n \to \infty$. This shows that $g_{\phi,jk} (t) \geq 0$ for almost all $t\in [0,\infty)$ and, thus, completes the proof.
\end{proof}

\subsection*{Acknowledgments}
This work was supported by the Danish Council for Independent Research (grants 4002-00003 and 9056-00011B).

\bibliographystyle{chicago}

\end{document}